\documentclass[11pt]{amsart}
\usepackage[left=2cm,top=2cm,right=3cm,nofoot]{geometry}
\usepackage{amsmath,mathtools}%
\usepackage{amsfonts}%
\usepackage{amssymb}%
\usepackage{graphicx}
\usepackage{esint}
\usepackage{hyperref}
\usepackage{enumitem}
\usepackage{accents}
\usepackage{etoolbox}
\patchcmd{\subsection}{-.5em}{.5em}{}{}
\newtheorem{theorem}{Theorem}[section]
\theoremstyle{plain}

\newtheorem{corollary}[theorem]{Corollary}

\newtheorem{definition}[theorem]{Definition}

\newtheorem{lemma}[theorem]{Lemma}

\newtheorem{remark}[theorem]{Remark}

\numberwithin{equation}{section}
\theoremstyle{plain}

\usepackage{etoolbox}
\AtEndEnvironment{proof}{\setcounter{claim}{0}}

\begin{document}

\begin{abstract} On a closed Riemannian manifold $(M^n ,g)$ with a proper isoparametric function $f$ we consider the
equation $\Delta^2 u -\alpha \Delta u +\beta u  = u^q$, where $\alpha$ and $\beta$ are positive constants 
 satisfying that
$\alpha^2 \geq 4 \beta$. We let ${\bf m}$ be the minimum of the dimensions of the focal varieties of $f$ and
$q_f = \frac{n-{\bf m}+4}{n-{\bf m}-4}$, $q_f = \infty$ if $n\leq {\bf m}+4$. We prove the 
existence of infinitely many nodal solutions of the equation assuming that $1<q<q_f$. The solutions are $f$-invariant.
To obtain the result, first we prove a $C^0-$estimate for positive $f$-invariant solutions of the equation. Then we prove the existence of mountain pass solutions with arbitrarily large energy.
\end{abstract}

\title[Nodal solutions to Paneitz-type equations]{Nodal solutions to Paneitz-type equations}

\author{ Jurgen Julio-Batalla}
\address{ Universidad Industrial de Santander, Carrera 27 calle 9, 680002, Bucaramanga, Santander, Colombia}
\email{ jajuliob@uis.edu.co}
\thanks{  The first author was supported by project 3756 of Vicerrector\'ia de Investigaci\'on y Extensi\'on of Universidad Industrial de Santander}

\author{Jimmy Petean}\thanks{The second author was supported by grant A1-S-45886 of Fondo Sectorial  SEP-CONACYT}  
\address{Centro de Investigaci\'{o}n en Matem\'{a}ticas, CIMAT, Calle Jalisco s/n, 36023 Guanajuato, Guanajuato, M\'{e}xico}
\email{jimmy@cimat.mx}

\maketitle

\section{Introduction}
Let $(M^n ,g)$  be a closed Riemannian manifold of dimension $n\geq 5$. We will consider the
{\it Paneitz type equation}

\begin{equation}\label{Paneitz}
\Delta^2 u -\alpha \Delta u +\beta u  = u^q  ,
\end{equation}

\noindent
 with $\alpha , \beta \in \mathbb{R}_{>0}$ and  $q>1$. In this article 
$\Delta= \Delta_g =div_g \nabla$ will denote  the non-positive
Laplace operator. We let
$p^* = (n+4)/(n-4)$. The equation (\ref{Paneitz})
is called critical if $q=p^*$, subcritical if  $q < p^*$ and supercritical if  $q > p^*$.

Paneitz type equations, in particular the critical equation,  appeared in the context of  the constant $Q$-curvature 
problem in conformal geometry. 
The $Q$-curvature of a Riemannian manifold, and the related Paneitz operator,  
were introduced by S. Paneitz \cite{Paneitz} and T. Branson \cite{Branson}
in their study of conformally invariant operators. It is defined by

\begin{equation}\label{Qcurv}
Q_g:=-\frac{1}{2(n-1)} \Delta_{g} sc_{g}-\frac{2}{(n-2)^{2}}\left| \mathrm{Ric}_{g}  \right|^{2}+\frac{n^{3}-4 n^{2}+16 n-16}{8(n-1)^{2}(n-2)^{2}} sc_{g}^{2},
\end{equation}

\noindent
where $sc_{g}$ is
the scalar curvature of $g$ and $\mathrm{Ric}_{g}$ is the Ricci tensor.

The Paneitz operator, $P_g$, is defined by 

\begin{equation}\label{opPan}
P_{g} \psi:=\Delta^{2} \psi+\frac{4}{n-2} \mathrm{div}_{g} \left(    \mathrm{Ric}_{g} \left(\nabla \psi, e_{i}\right) e_{i}\right)-\frac{n^{2}-4 n+8}{2(n-1)(n-2)} \mathrm{div}_{g}\left(sc_{g} \nabla \psi\right)+\frac{n-4}{2} Q_{g} \psi,
\end{equation}
 
\noindent
where $\{e_i\}_i$ is a $g-$orthonormal frame. The fundamental characteristic of the Paneitz operator is that it is conformally invariant: 
if we write a metric $h$ conformal  to $g$ as $h=u^{\frac{4}{n-4}} g$, for a positive function $u: M \rightarrow \mathbb{R}$, then

\begin{equation*}
P_{h} \psi=u^{-\frac{n+4}{n-4}} P_{g}(u \psi).
\end{equation*}
Then the expression for  the $Q$-curvature of $h$ in terms of 	$g$ and $u$  is given by:
\begin{equation*}
Q_{h}=\frac{2}{n-4} u^{-\frac{n+4}{n-4}} P_{g}(u).
\end{equation*}

It follows that  the problem of  finding a conformal metric $h=u^{\frac{4}{n-4}} g$ with constant $Q$-curvature 
$\lambda \in \mathbb{R}$ 
is equivalent to finding a positive solution $u$ of the fourth order {\it Paneitz-Branson equation}:
\begin{equation}\label{ConstantQcurv}
P_{g} u=\lambda u^{\frac{n+4}{n-4}}.
\end{equation}

It follows easily from a direct computation   that Einstein metrics have constant $Q$-curvature given by

$$Q_g = \frac{n^2 -4}{8n(n-1)^2} sc_{g}^2 .$$

Also note that if $(M,g)$ is an Einstein manifold with positive scalar curvature  then the Paneitz operator takes the form 

$$P_g (u) = \Delta^2 u -\alpha \Delta u +\beta u ,$$

\noindent
where $$\alpha=\dfrac{n^2-2n-4}{2n(n-1)}sc_g >0 ; \quad\beta=\frac{n-4}{2}Q_g >0.$$

We also point out  that in this case

$$\alpha^2 - 4\beta = \frac{4sc_{g}^2 }{(n-1)^2} >0.$$ 

\bigskip

The general Equation  (\ref{Paneitz}) was considered by Z. Djadli, E. Hebey and M. Ledoux in \cite{Hebey}, where
important results were obtained  in the critical case.  The study of positive
solutions of Equation (\ref{ConstantQcurv}), which is equivalent to the problem of finding conformal constant $Q$-curvature metrics, 
has received a lot of attention. About the problem of existence of positive solutions we mention for instance the
articles by F. Robert \cite{Robert}, J. Qing and D. Raske \cite{Qing}, by  M. J. Gursky, F. Hang and
Y.-J. Lin \cite{Gursky}, by M. J. Gursky and A. Malchiodi \cite{Malchiodi} and by F. Hang and  P. C. Yang
\cite{Hang, Yang}. Multiplicity of solutions and compactness of the space of positive solutions was treated 
for instance by E. Hebey and F. Robert in \cite{Hebey2}, by A.  Malchiodi in \cite{Malchiodi2},
by G. Li  in \cite{Li}, by Y. Y. Li and J. Xiong in  \cite{YanYanLi}, by R. Bettiol, P. Piccione and Y. Sire in
\cite{Bettiol} and by J. Wei and C. Zhao in \cite{Wei}.

\bigskip

In this article we  will assume that there exists an {\it isoparametric function} $f$ on $(M,g)$ and study solutions of Equation
(\ref{Paneitz}) which are constant along the level sets of $f$. We recall that a function $f:M \rightarrow [t_0 , t_1 ]$ is
called {\it isoparametric} if there are functions ${\bf a}, {\bf b} : [t_0 , t_1 ] \rightarrow \mathbb{R}$ such that 
$\| \nabla f \|^2 = {\bf a} \circ f$ and $\Delta f = {\bf b} \circ f$. Isoparametric functions on general Riemannian manifolds 
were introduced by Q-M.  Wang in \cite{Wang}, following the classical history of isoparametric functions on 
space forms. The simplest examples of isoparametric functions are functions invariant by a cohomogeneity-one
isometric action on $(M,g)$: these are called homogeneous isoparametric functions. But there are many examples of
isoparametric functions which are not homogeneous. The first examples were constructed by  H. Ozeki, M. Takeuchi
in \cite{Ozeki, Takeuchi}. These examples were later generalized by D. Ferus, H. Karcher and H. F. M\"{u}nzner
in \cite{Ferus}. In a more general setting, C. Qian and  Z. Tang proved in \cite{Qian} 
that given a Morse-Bott function
on a closed manifold $M$ with certain conditions on the set of critical points,  there is a Riemannian metric $g$ on $M$
such that $f$ is an isoparametric function on $(M,g)$. 

\medskip

There are many general results known about the structure of 
isoparametric functions on Riemannian manifolds. It was
proved by Q-M. Wang in \cite{Wang} that
the only critical values of an isoparametric function are its
minimum and its maximum, $t_0$ and $t_1$. The critical level sets $M_0 = f^{-1} (t_0 )$ and $M_1 = f^{-1} (t_1)$ are called
the {\it focal varieties} of $f$, and  it was also proved in \cite{Wang} that they are smooth
submanifolds of $M$. The isoparametric function is called {\it proper} if the level sets of
$f$ are connected, which is equivalent to saying that the dimension of both focal varieties is at least $n-2$. 
In this article we will use the following notation: for $i=0, 1$, we let 
$m_i$ be the dimension of $M_i$ and we denote by  ${\bf m}= \min \{ m_0 , m_1 \}$.

\medskip

Given an isoparametric function $f$ we will say that a function is {\it f-invariant} if it is constant on the level sets 
of $f$.
We will study  multiplicity results for  $f$-invariant solutions of Equation (\ref{Paneitz}).

We define the energy of a solution $u$ of  Equation (\ref{Paneitz}) as

$$E(u)=\int\limits_{M}|u|^{q+1}dv_g,$$
where $dv_g$ is the volume element of $(M,g)$.

\medskip

We  first prove that the space of positive, $f$-invariant solutions, is compact under a condition on the exponent $q$. 
As above, we let ${\bf m}= \min \{ m_0 , m_1 \}$ be the minimum of the
 dimension of the focal varieties, and we let $q_f = (n-{\bf m}+4)/(n-{\bf m}-4)$, if $n- {\bf m} -4 >0$,
$q_f = \infty$ otherwise.
Compactness implies in particular that there is an upper bound
for the energy of positive solutions. We prove: 

\begin{theorem}
Let $f$ be a proper isoparametric function on $(M,g)$. Assume that  $1<q<q_f$. Then for any $0 <a < b $, the set $S=S_{a , b} = \{(u,\alpha , \beta ) : u$ is a positive , $f$-invariant solution of Equation (\ref{Paneitz}) in $C^{4,\gamma}(M)$
and $\alpha , \beta  \in [a,b]$ , with $\alpha^2 \geq 4\beta  \} $ is compact.
\end{theorem}\label{compactness}

Then we apply the classical theory of mountain pass critical points of A. Ambrosetti and P. H. Rabinowitz
\cite{Ambrosetti} to prove:

\begin{theorem}Let $f$ be a proper isoparametric function on $(M,g)$. Assume that  $1<q<q_f$.
Given $\alpha, \beta >0$ there is a sequence of solutions of Equation  (\ref{Paneitz})
with unbounded energy. 
\end{theorem}

The previous results imply:

\begin{corollary}\label{infinite} Let $f$ be a proper isoparametric function on $(M,g)$. Assume that  $1<q<q_f$.
Given $\alpha, \beta >0$ with $\alpha^2 \geq 4\beta$,  there exist infinitely many nodal $f-$invariant solutions of 
Equation (\ref{Paneitz}).
\end{corollary}

The existence of infinite nodal solutions to the critical equation in the round sphere was proved 
by  T. Bartsch, M. Schneider and T. Weth in 
\cite{Bartsch} and by N. Saintier in \cite{Saintier}, also
using the mountain pass theorem. In this case the fact that infinitely many of the solutions are nodal follows
from the classification of positive solutions of the Paneitz-Branson equation on the sphere obtained by
C.-S. Lin in \cite{Lin} and by J. Wei and X. Xu in \cite{class}.

\medskip

Note that if ${\bf m}>0$ then $q_f >p^*$ and the previous results apply to critical and supercritical equations. 
In particular, if $(M,g)$ 
is a positive Einstein manifold which admits an isometric cohomogeneity one action with all orbits of positive dimension,
then the previous results apply to the Paneitz-Branson equation on $(M,g)$. Therefore as
a corollary we have:

\begin{corollary}
If $(M,g)$ is a positive Einstein manifold with a cohomogeneity one action with positive dimensional orbits
then the Paneitz-Branson equation (\ref{ConstantQcurv}) on $(M,g)$ has infinitely many  nodal solutions.
\end{corollary}

This is in contrast with the result of J. Vetois in \cite{Vetois}, that if $(M,g)$ is a closed Einstein manifold
with positive scalar curvature, 
different from the round sphere, then
the only positive solutions of the Paneitz-Branson equation on $(M,g)$ are the constant functions.

\bigskip

In Section 2 we will discuss general facts on isoparametric functions and prove Theorem 1.1. In Section 3 we
will apply the theory of mountain pass critical points to prove Theorem 1.2.

\section{Isoparametric functions and compactness}

Let $(M^n , g)$ be a closed Riemannian manifold of dimension $n$ and assume that there is a proper isoparametric
function $f:M \rightarrow [t_0 , t_1 ]$. 
In this section we will discuss solutions of the Paneitz-type equation (\ref{Paneitz}) which are constant on the 
level sets of $f$ and prove Theorem 1.1.

The classical theory of isoparametric functions on general Riemannian manifolds was introduced by Q-M. Wang in \cite{Wang}. 
It is known that an 
isoparametric function $f$ only has singular values at $t_0=\min f$ and $t_1=\max f$. 
The critical level sets $M_0 = f^{-1}(t_0)$ and $M_1 = f^{-1}(t_1)$ are smooth submanifolds, which are usually 
called the {\it focal varieties} of $f$.
Moreover, all the 
regular level sets of $f$, $f^{-1}(t)$ for $t\in (t_0 , t_1 )$, have constant mean curvature and
are smooth tubular hypersurfaces around 
each of the focal submanifolds $M_0$ and $M_1$. Let $d_g$ denote
the distance function on $(M,g)$. Let $D= d_g (M_0 , M_1 )$ be the distance between the focal varieties. 
We denote by ${\bf d}:M\rightarrow [0,D]$ the distance to  the  focal variety $M_0$, ${\bf d}(x) = d_g (x, M_0 )$. 
Note that it follows from
the previous comments that the level sets of $f$ and ${\bf d}$ are the same.

\begin{definition}
A function $u:M\rightarrow \mathbb{R}$ is called $f$-invariant if $u(x)=\phi({\bf d}(x))$ for some  function $\phi:[0,D ]\rightarrow \mathbb{R}$.
\end{definition}

Note that $f$-invariant functions are precisely the functions which are constant along the level sets of $f$, and that 
an $f$-invariant function $u$ is obviously determined by the corresponding function $\phi$ such that $u= \phi \circ 
{\bf d}$.

The Paneitz-type Equation (\ref{Paneitz}) for an $f$-invariant function $u= \phi \circ 
{\bf d}$ is equivalent to a fourth-order  ordinary differential equation on $\phi$. 
For $t\in (0,D)$ let $h(t)$ denote the (constant) mean curvature of ${\bf d}^{-1} (t)$. 
Then by a direct computation we have that for $x\in M-(M_0 \cup M_1 )$, 

$$\Delta (\phi \circ{\bf d} ) (x) = (\phi '' +h\phi ' )\circ {\bf d} (x) .$$

Then we can see (details can be found in  \cite[Lemma 2.4]{BQC}) that  $u$ solves the Paneitz type Equation
 (\ref{Paneitz}) if and only if $\phi$ solves

\begin{equation}\label{PaneitzODE}
\phi  '''' + 2h \phi  ''' +(2h' +h^2 -\alpha ) \phi '' +(h''+hh ' -\alpha h )\phi  ' + \beta \phi = \phi^ q .
\end{equation}

The equation is
considered in the open interval $(0,D)$. For the function $u$ to be defined on all of $M$, $\phi$ must solve the equation on $[0,D]$, with the
appropriate boundary conditions; but the boundary conditions will play no role in the present article.

\medskip

To study  Equation (\ref{Paneitz}) around the focal varieties of $f$ we will consider Fermi coordinates. Namely, 
we first consider geodesic coordinates $((x^1,\cdots,x^{m_0}),\Omega')$ centered at $p$ in $\Omega'$, where $\Omega'$ is an open subset
of $M_0$. And then we consider geodesic coordinates in the directions normal to $M_0$. Explicitly, we have
coordinates $((x^1,\cdots,x^n),\Omega)$ where  $(x^{m_0+1},\cdots,x^n)$ are defined as follows: for all $\bar{p}\in \Omega'$ and $r_{i's}$ small enough,
$$x^j\left(exp_{\bar{p}}\left(\sum\limits_{i=m_0+1}^{n}r_iE_i(\bar{p})\right)\right)=r_j,$$
where $E_{m_0+1},\cdots, E_n$ are orthonormal sections in normal direction of $M_0$ which are parallel with respect to the normal connection along geodesics in $M_0$ starting at $p$ in $\Omega'$. More details can be found
for instance in  \cite{Tubes}.

In these coordinates, the distance function $\textbf{d}:=d|_{\Omega}$ takes the form  $$\textbf{d}(x)=\sqrt{\sum_{i=m_0+1}^n(x^i)^2}.$$
Also, we have that $g(\partial x^i,\partial x^j)|_{\Omega'}=\delta_{ij}$ for $i,j\in\{m_0+1,\cdots,n\}.$

\bigskip

We are now ready to give the proof of Theorem 1.1:

\begin{proof}
We will prove first that functions in $S$ are uniformly $C^0$-bounded. Assume that there exists a sequence $U_i=(u_i,\alpha_i,\beta_i)$ in $S$ such that $max\; u_i\rightarrow\infty$.

We will apply scaling blow-up arguments in a similar way used in the case of second order elliptic equations.

Let $p_i \in M$ be such that $u_i(p_i)=max\;u_i$. Since $M$ is compact we can assume that the sequence converges 
to some point $p\in M$. We can moreover also assume that $\alpha_i \rightarrow \alpha \in [a,b]$ and $\beta_i \rightarrow \beta \in [a,b]$.

We divide the proof in two cases: 

\vspace{.3cm}

1) When the  point $p$ is in a focal variety of $f$ 

\vspace{.3cm}

2) When $p$ is in a regular level set of $f$.

\vspace{.5cm}

Case 1. Assume for instance  that  $p\in M_0$.

Consider Fermi coordinates $\varphi=(x^1,\cdots,x^n)$ around $M_0$ centered at $p$. For $i$ large enough, let $x'_i$ be the $n-m_0$ coordinates of $p_i$ in the normal directions to $M_0$ i.e. $$\varphi(p_i)=(x^1(p_i),x^2(p_i),\cdots,x^{m_0}(p_i),x'_i).$$

We consider the functions $w_i$ defined by $$w_i((x-x'_i)/\lambda_i)=\lambda_i^{4/(q-1)}u_i(\varphi^{-1}(x)),$$

\noindent
where the constants $\lambda_i$ are defined so that $$w_i(0)=\lambda_i^{4/(q-1)}u_i(\varphi^{-1}(x'_i))=\lambda_i^{4/(q-1)}u_i(p_i)=1.$$

Note that we have that $\lambda_i\rightarrow 0$ and $max\;w_i=w_i(0)=1$.

Since the function  $u_i$ is $f-$invariant, there exists $\phi_i$ such that 
$$u_i(x)=\phi_i\left(\sqrt{\sum_{i=m_0+1}^n(x^i)^2}\right)$$
in the coordinates $\varphi$. In particular $u_i$  only depends of the components $(x^{m_0+1},\cdots,x^n)$ in these coordinates. Therefore, 
the function  $w_i$ is defined on the $(n-m_0)-$Euclidean ball $B(0,R/2\lambda_i)$ of center 0 and radius $R/2\lambda_i$, for some fixed $R>0$.

Since $u_i$ is a solution of Equation (\ref{Paneitz}), with $\alpha = \alpha_i$, $\beta=\beta_i$, it follows by a direct computation 
that $w_i$ satisfies the equation
$$\Delta^2_{g_i}w_i-\lambda_i^2\alpha_i\Delta_{g_i}w_i+\lambda_i^4\beta_iw_i=w_i^q,$$
where $g_i(y)=((\varphi^{-1})^*g)(\lambda_iy+x'_i)$.

Moreover each $w_i(y)$ is a $g_i-$radial function on $(B(0,R/2\lambda_i),g_i)$.

On the other hand, the condition $\alpha_i^2-4\beta_i  \geq 0$ (which in particular says that $(\lambda_i^2\alpha_i)^2-4\lambda_i^4\beta \geq 0$) implies that the function $w_i$ satisfies

$$(-\Delta_{g_i}+c_1(i)\lambda_i^2)\circ(-\Delta_{g_i}+c_2(i)\lambda_i^2)w_i=w_i^q,$$

where $$c_1(i)=\alpha_i/2+\sqrt{\alpha_i^2/4-\beta_i} >0 \quad\text{and}\quad c_2(i)=\alpha_i/2-\sqrt{\alpha_i^2/4-\beta_i} >0.$$
The coefficients in this equation are uniformly bounded and continuous in $B(0,R/{2\lambda_i})$, so we can apply 
classical regularity theory. Indeed, by $L^p$ theory of elliptic operators the functions $w_i$ are uniformly bounded in the Sobolev space $L^k_2(B(0,R'))$ for any $R'<R/(2\lambda_i)$.

Taking $k>n$, it follows from Morrey's inequality that the sequence $w_i$ is uniformly bounded in $C^{0,\gamma}(B(0,R'))$.

Therefore  elliptic regularity theory implies that $$\|w_i\|_{C^{4,\gamma}(B(0,R'))}\leq c,$$
for any $R'<R/2\lambda_i$ and  constant $c$ depending only on $R'$.

Using Arzela-Ascoli theorem and a standard diagonal argument we obtain a subsequence $w_i$ and $w\in C^4(\mathbb{R}^{n-m_0})$ such that $w_i\rightarrow w$ locally in $C^4(\mathbb{R}^{n-m_0})$.

Since $\lambda_i \rightarrow 0$, it follows that the limit function $w$ is a nonnegative solution of $$\Delta_{g_0}^2w=w^q,$$
where $g_0=\lim\limits_{i\rightarrow  \infty}(\varphi^{-1*}g)(\lambda_iy+x'_i)$. 

It then follows from the definition of Fermi coordinates that  the metric $g_0$ is the Euclidean metric in $\mathbb{R}^{n-m_0}$, and
therefore $w$ is a radial function in $\mathbb{R}^{n-m_0}$, with $w(0)=1$.

Note also that by hypothesis $q<\frac{n-{\bf m}+4}{n-{\bf m}-4}\leq\frac{n-m_0+4}{n-m_0-4}.$

But it was proved by M. Lazzo and P. G. Schmidt in Theorem 1.1 in \cite{subcritical} that  under the
previous conditions the function $w$ must be oscillatory. But this is a contradiction since we had seen that the function $w$ was non-negative.

\vspace{.5cm}

Case 2. Assume now that $p\in f^{-1}(s)$ for some regular value $s \in (t_0 , t_1 )$.

This case is a more elementary version of Case 1. Let $s_i = f(p_i )$, so $s_i \rightarrow s$. Note that we can assume that 
$s_i \in (t_0 , t_1 )$ is a regular value. Moreover, we can assume that there exists $\varepsilon >0$ such that for all $i$,
$(s_i -\varepsilon , s_i + \varepsilon ) \subset (t_0 , t_1 )$. We write $u_i = \phi_i \circ {\bf d}$ as before. 
Then $\phi_i$ is a solution of Equation 
(\ref{PaneitzODE}):

$$\phi  '''' + 2h \phi  ''' +(2h' +h^2 -\alpha ) \phi '' +(h''+hh ' -\alpha h )\phi  ' + \beta \phi  =\phi^ q ,$$ 

\noindent
where the function $h$ is smooth on $(0,D)$.

Now we consider the function 

$$\varphi_i (t) = \frac{1}{m_i} \phi \left( \frac{t-s_i}{m_i^{\frac{q-1}{4}}} \right) .$$

We consider the function $\varphi_i$ defined on $(-\varepsilon m_i , \varepsilon m_i )$. It satisfies the equation

$$\varphi  '''' + \frac{2h}{m_i^{\frac{q-1}{4}}}  \varphi  ''' +
\frac{(2h' +h^2 -\alpha_i )}{m_i^{\frac{2(q-1)}{4}}}  \varphi '' +
\frac{(h''+hh ' -\alpha_i h )}{m_i^{\frac{3(q-1)}{4}}} \varphi  ' + \frac{\beta_i}{m_i^{q-1}}  \varphi =\varphi^q,$$ 

Then the sequence $\varphi_i$ converges $C^4$-uniformly on compact sets of $\mathbb{R}$ to a nonnegative function $\varphi$ 
which verifies $\varphi (0) =1$ and satisfies the equation

$$\varphi '''' = \varphi^q .$$

It is easy to see that no such function can exist. Therefore we reach again a contradiction. Note that in this case we only
use that $q>1$.

\vspace{.5cm}

Therefore, we conclude that $S$ is a $C^0-$bounded.

Now, for any sequence $w_i$ of functions on $S$ we have a uniform $C^0$-bound for $w_i$. Again, we can split the linear part of Equation (\ref{Paneitz}) as the composition of two second order
operators: 

$$(-\Delta +c_1(i) )\circ  (-\Delta  +c_2(i)   )w_i=w_i^q,$$

\noindent
where we recall that $c_1 (i), c_2 (i) >0$. Let $z_i =  (-\Delta  +c_2(i)   )w_i$. We have that  $(-\Delta +c_1(i) ) (z_i)$ is 
uniformly $C^0$-bounded.
Then it follows from regular elliptic theory, that the sequence $z_i$ is uniformly $C^{2, \gamma}$-bounded. And then again by
regular elliptic theory it follows that the sequence $w_i$ is bounded in $C^{4,\gamma}$. 
Therefore, it converges (up to a subsequence) to a nonnegative function $W$  which is solution of Equation (\ref{Paneitz} with
$\alpha = \lim \alpha_i$, and $ \beta =\lim \beta_i$. Note that $\alpha , \beta  \in [a , b]$ , and  $\alpha^2 \geq 4\beta $. 

Note that it follows from Equation (\ref{Paneitz}) that since the functions $w_i$ are positive the maximum of $w_i$
must be at least $\beta^{\frac{1}{q-1}}$ (if the maximum is equal to this value then $w_i$ must be constant). It follows
that $W$ cannot be the constant 0.  Let 

$$c_1 =\alpha /2+\sqrt{\alpha^2/4-\beta} >0 \quad\text{and}\quad c_2 =\alpha/2-\sqrt{\alpha^2/4-\beta} >0,$$

\noindent
so that $(-\Delta +c_1 )\circ  (-\Delta  +c_2   ) W=W^q$. Then as before we see by elliptic theory that 
$ (-\Delta  +c_2   ) W \geq 0$ and then $W$ must either be strictly positive or $W \equiv 0$. 
Therefore $W$ is strictly positive and $W\in S$.

\end{proof}

\begin{remark} If ${\bf m}>0$ then the exponent $1<q<(n-{\bf m}+4)/(n-{\bf m}-4)$ can be critical and supercritical with respect to the Sobolev embedding $$L^2_2(M^n)\subset L^{2n/(n-4)}(M^n).$$

\end{remark}

\section{Mountain pass critical points}

In this section we will prove Theorem 1.2. We assume that $f$ is a proper isoparametric function 
on $(M^n ,g)$ and let ${\bf m}$ be the minimum of the dimensions of the focal submanifolds of $f$.
We consider on $M$  the  Paneitz-type equation  (\ref{Paneitz}) $$\Delta^2u-\alpha\Delta u +\beta u=u^q,$$
where $1<q< q_f = (n-{\bf m}+4)/(n-{\bf m}-4)$ and $\alpha,\beta>0$.

We call the linear part of the equation $P_{(\alpha,\beta)}$, the Paneitz operator, and we point out that it
is a coercive operator.

We will find solutions of the Paneitz-type equation in the space of $f-$invariant functions (see Section 2). We let
$S_f$ be the family of $f$-invariant functions and denote by $L^2_{2,f}(M) = S_f \cap L^2_{2}(M)$.
Also for any $r>0$ let $L^r_f (M) = S_f \cap L^r (M)$.

\medskip

For the proof of Theorem 1.2 we will use the classical min-max method of Ambrosetti-Rabinowitz  applied 
to the functional $I:L^2_{2,f}(M)\rightarrow\mathbb{R}$ defined as
$$I(u)=\frac{1}{2}\int\limits_{M}\left((\Delta u)^2 +\alpha |\nabla u|^2+\beta u^2\right)dv_g-\frac{1}{q+1}\int\limits_M |u|^{q+1}dv_g. $$

We point out that $L^2_{2,f}(M)$ is a closed subspace of $L^2_{2}(M)$, and it is $\Delta$-invariant. It follows that critical
points of $I$ are solutions of the Paneitz type equation (\ref{Paneitz}). 

\medskip

To obtain mountain pass solutions it will be fundamental that we have a compact embedding $L^2_{2,f}(M^n)\subset L_f^{q+1}(M^n)$.
In the case when $f$ is invariant by an isometric action on $(M,g)$ and the level sets of $f$ are the orbits of the action, this follows from the general
results of E. Hebey and M. Vaughon in \cite{Hebey3}. In the case of general isoparametric functions G. Henry proved in \cite{Henry} a
similar result for the inclusion of  $L^2_{1,f}(M^n)$ in $L_f^{q+1}(M^n)$ when $q<(n-{\bf m}+2 )/(n-{\bf m} -2)$.

\medskip

We first recall a few facts about isoparametric functions.
It is well-known that $M^n$ can be identify as a union of two disk bundles $D_0, D_1$, each one over $M_0$ and $M_1$ respectively. More specifically,  for $i\in\{0,1\}$ let $exp_{M_i}$ be the normal exponential map of $M_i$ in $M$. Following the notation in \cite{Miyaoka} the manifold $M^n$ is diffeomorphic to 
$$N_{\leq a}M_0\cup N_{\leq a}M_1,$$
where $$N_{\leq a}Q=\{exp_{Q}(s\eta)/\quad|\eta|=1, s<a\}\cup \{\exp_{Q}(a\eta)/\quad|\eta|=1\},$$
and $2a=d_g(M_0,M_1)$.

From the definition of a disk bundle, we can choose a coordinate system $(U_j,\varphi_j)$ on $N_{\leq a}M_0$ (also analogous one on $N_{\leq a}M_1$) such that  $U_j=\pi^{-1}(U'_j)$ for a finite cover $\{U'_j\}$ of $M_0$ and each $\varphi_j$ is a diffeomorphism defined  by $$\varphi_j:U_j\rightarrow U'_j\times D^{n-m_0}_0(a),$$ where $D^{n-m_0}_0(a)$ is the disk of radius $a$ in the normal bundle of $M_0$.

Without loss of generality, we cover $M^n$ by a finite number $m$ of these type of charts with a uniform bound for the metric tensor $g$ i.e. there exists a constant $c>1$ with $$c^{-1}I\leq g^l_{ij}\leq cI\quad\text{for}\quad l\in\{1,\cdots,m\}.$$

With the previous preliminaries we can now prove:

\begin{lemma} If $q<q_f$ 
the embedding $$L^2_{2,f}(M^n)\subset L_f^{q+1}(M^n)$$
is compact.
\end{lemma}
\begin{proof}
Let $s=n-m_0$ (recall $m_0\leq m_1$).

For any $f-$invariant function  $u$ we have that $u$ only depends of the normal coordinates on $N_{\leq a}M_0$ and $N_{\leq a}M_1$. In particular, using the previous charts $(U_j,\varphi_j)$ on $N_{\leq a}M_0$, the function $u$ only depends on $D^s_0(a)$ (similarly, $u$ only depends on $D_0^{n-m_1}(a)$ in $N_{\leq a}M_1$).

Thus we have positive constants $k$, $k_1$ such that 
\begin{eqnarray*}
\int\limits_{U_l} u^{q+1} dv_g&=\int\limits_{U'_l\times D_0^s(a)}u^{q+1}\sqrt{det g^l_{ij}}dxdy\\
&=k\int\limits_{D_0^s(a)}u^{q+1}\sqrt{det g^l_{ij}}dy\\
&\leq k_1\int\limits_{D_0^s(a)}u^{q+1}dy.
\end{eqnarray*}
It follows that there is a positive constant $k_2$ such that  $$|u|_{L^{q+1}(U_l)}\leq k_2|u|_{L^{q+1}(D_0^s(a) )}.$$
Since $g^l_{ij}$ is bounded below, we obtain that $$|u|_{L^2_2(U_l)}\geq k_3 |u|_{L^2_2(D_0^s(a))},$$

\noindent
for some positive constant $k_3$.

Since $q<(s+4)/(s-4)$ (or equivalently $\frac{1}{2}-\frac{2}{s}<\frac{1}{q+1}$) by the Sobolev inequality in $D_0^s(a)\subset\mathbb{R}^{n-m_0}$ we have that  $$|u|_{L^{q+1}(D_0^s(a)) }\leq k_4 |u|_{L^2_2(D_0^s(a))},$$

\noindent
for some positive constant $k_4$.

And therefore we have a positive constant $C(l)$ such that  $$|u|_{L^{q+1}(U_l)}\leq C(l) |u|_{L^2_2(U_l)}. $$

Then we can take  a partition of unity on $M$ subordinate to $\{U_l\}_l$ and sum up on $l$ 
to conclude that there exists a constant $C>0$ such that

\begin{equation}\label{A}
|u|_{L^{q+1}(M)}\leq C |u|_{L^2_2(M)}
\end{equation}

\noindent
for  any $u\in L^{2}_{2,f} (M)$. 

Finally, to obtain compactness of the embedding, we follow the classical argument of Kondrakov's theorem on closed manifolds. Indeed let $(u_i)$ be a bounded sequence of functions in $L^2_{2,f}(M^n)$. Assume that $(\delta_l)$ is a partition of unity on $M$ subordinate to $(U_l,\varphi_l)$.  It follows from compactness of 
the embedding of $L^{q+1}(D_0^s(a))$ in $L_2^2(D_0^s(a))$ that the sequence $(\delta_l u_i)_i$ has a convergent subsequence in $L^{q+1}(D_0^s(a))$. In particular $(\delta_lu_i)_i$ has a Cauchy subsequence in $L^{q+1}(U_l)$. Repeating the process for each $l$, we have that $(u_i)$ has a subsequence $\bar{u}_i$ such that 
$(\delta_l\bar{u}_i)_i$ is a Cauchy sequence in $L^{q+1}_f(M^n)$ for each $l$. Since $|\bar{u}_i-\bar{u}_j|\leq \sum\limits_{l=1}^r|\delta_l(\bar{u}_i-\bar{u}_j)|$ the sequence $(\bar{u}_i)$ is Cauchy in $L^{q+1}_f(M^n)$ and hence convergent.  
 
\end{proof}

It follows in a standard way from the previous lemma that the functional $I$ satisfies the Palais-Smale condition. 
It is also clear
that $I$ is even, $I(0)=0$.

\medskip

Since $\alpha$, $\beta$ are positive, the operator $P_{\alpha,\beta}$ is coercive and there exists a positive constant $D$ such that 

$$\int_M u P_{\alpha,\beta}(u)dv_g\geq D |u|^2_{L_{2}^2},$$

\noindent
for any $u \in L^2_2 (M)$.

Also, by (\ref{A}) we have a positive constant $E$  such that  if $u\in L^2_{2,f}$ 

$$\int_M u^{q+1}dv_g\leq E |u|^{q+1}_{L^2_{2}}.$$

Hence $$I(u)\geq \frac{D}{2} |u|^2_{L_2^2}- \frac{E}{q+1} |u|^{q+1}_{L^2_{2}} .$$

For any positive constant $\rho$ let $B(\rho )$ be the ball of radius $\rho$ centered at $0$ in  $L_{2,f}^2(M)$.
For a function $u\in L^2_{2,f}$ with $|u|_{L^2_{2}}=1$, we have that  

$$I(tu) \geq \frac{C}{2} t^2 - \frac{D}{q+1} t^{q+1}$$ 

\noindent
for any $t>0$. Since $q>1$ it follows that for $\rho >0$ small enough we have that $I(u) >0$ if $u\in B(\rho ) -\{ 0 \}$ and there is
a positive lower bound for $I$ on $\partial B(\rho )$.

\medskip

Let $E_m$ be a subspace of $L^2_{2,f}$ of  (finite) dimension $m$. Since the set of
$u\in E_m$ such that $\int_M uP_{\alpha,\beta}(u)dv_g=1$ is compact there is  a positive constant $C$ 
such that for all such $u$   $$C \leq |u|^{q+1}_{L_f^{q+1}}.$$

Then for any  $u \in E_m$ with $\int_M uP_{\alpha,\beta}(u)dv_g=1$ and for any $t>0$ we have that
$$I(tu)\leq \frac{t^2}{2}-\frac{C \  t^{q+1}}{(q+1)}.$$

It easily follows from this inequality that $E_m\cap\{I\geq0\}$ is bounded. 

\bigskip

It follows from the previous lines that we can apply the mountain pass critical points theory of
 A. Ambrosetti and P. H. Rabinowitz,  Theorem 2.13 in \cite{Ambrosetti}, to obtain:

\noindent
{\it Let 

$$\Gamma^*=\{h:L_{2,f}^2(M)\rightarrow L_{2,f}^2(M)/ \;I(h(B(1))\geq 0\;,\;h\;\text{is an odd  homeomorphism}\}.$$

\noindent
Fix an orthonormal basis $\{ e_i \}_{i\geq 1}$ of  $L^2_{2,f}$  and let $E_m = \langle e_1 , ... , e_m \rangle$.

\noindent
Then 

$$c_m=\sup\limits_{h\in\Gamma^*}\inf\limits_{u\in E_m^{\perp}\cap \partial B(1)}I(h(u)).$$

\noindent
is an increasing sequence of critical values of $I$. }

\bigskip

Therefore, there exists a sequence $u_m\in L^2_{2,f}$ of weak solutions to the Paneitz-type equation associated to the mountain pass levels $c_m$.

Following regularity results (see for instance Proposition 3 in \cite{Esposito}) the functions $u_m$ are strong solutions and belong to $C^{4,\gamma}(M)$.

\bigskip

Finally we will prove that the solutions $u_m$ have large energy i. e. the increasing sequence $c_m$ goes to infinity. The proof of this
fact follows the same original ideas used in the proof
of Theorem 3.14 in \cite{Ambrosetti} for second order equation (and  for the Paneitz equation on the round sphere \cite{Bartsch, Saintier}) .

Let $$T=\left\lbrace u\in L^2_{2,f}-\{0\}\;/\; \frac{1}{2}\int_M uP_{\alpha,\beta}(u)dv_g=\frac{1}{q+1}\int_M |u|^{q+1}dv_g \right\rbrace,$$
and $$d_m=\inf\left\lbrace\int_M uP_{\alpha,\beta}(u)dv_g/\;u\in T\cap E_m^{\perp}\right\rbrace.$$

We will first prove that $\lim_{m\rightarrow \infty} d_m = \infty$. Assume that
$(d_m)$ is bounded.  Then there exists a positive constant $d$ and  a sequence $\{ u_m \}_{m\geq 1}$ with $u_m\in T\cap E_m^{\perp}$  such that  $$\int_M u_mP_{\alpha,\beta}(u_m)dv_g\leq d,\quad \forall m.$$

Since the operator  $P_{\alpha,\beta}$ is coercive the sequence $u_m$ is bounded in $L^2_{2,f}$, and therefore 
weakly convergent to some $u\in L^2_{2,f}$. But since $u$ must be orthogonal to $E_m$ for every $m\geq 1$ it follows that we must have
$u=0$.

And since the embedding $L^2_{2,f}\subset L_f^{q+1}$ is compact by Lemma 3.1, a subsequence satisfies that
$u_m\rightarrow 0$ strongly in $L^{q+1}_f$.

It is  also proved in Lemma 3.1, (\ref{A}),  that  we have a positive constant $K$ such that 
for any $u\in L^2_{2,f} (M) - \{ 0 \}$, 
$$K<\dfrac{|u|_{L^2_{2}}}{|u|_{L^{q+1}}}.$$

It follows from the coercivity of $P_{\alpha,\beta}$ and the definition of $T$  that for all $m$, 
$$0<K\leq k_0 \dfrac{\left(\int_M u_mP_{\alpha,\beta}(u_m)dv_g\right)^{1/2}}{\left(\int_M u_mP_{\alpha,\beta}(u_m)dv_g \right)^{1/(q+1)}}.$$
Since $q+1>2$ this implies that the sequence $( \int_M u_mP_{\alpha,\beta}(u_m)dv_g )$ is bounded from 
below by a positive constant. Then it follows
from the definition of $T$ that  the sequence $|u_m|_{L_f^{q+1}}$ must be bounded below by a positive constant, 
which is a contradiction. 

\bigskip

Therefore $d_m\rightarrow\infty$.

\vspace{.5cm}

Pick constants $C_1 , C_2 >0$ such that for all $u\in L_2^2 (M)$,

$$C_1 |u|_{L^2_{2}} \leq    \int_M uP_{\alpha,\beta}(u)dv_g      \leq    C_2      |u|_{L^2_{2}} .$$

Pick also a constant $C>\sqrt{C_2}$.

For any $u\in E_m^{\perp}\cap \partial B(1)$ we have that 
$$I\left( \frac{\sqrt{d_m}}{C} u \right) =\frac{d_m}{2C^2}\int_M uP_{\alpha,\beta}(u)dv_g-\frac{1}{q+1}\cdot\left({\frac{\sqrt{d_m}}{C}}\right)^{q+1}|u|^{q+1}_{L^{q+1}}.$$

Note also that for each $u\in E_m^{\perp}\cap \partial B(1)$ there exists a unique positive number $a= a(u)$ such that $a(u)u\in T$. Explicitly

$$a=\left(    \frac{(q+1)\int_M uP_{\alpha , \beta} (u)dv_g }{ 2 \int_M |u|^{q+1}dv_g     }           \right)^{\frac{1}{q-1}}.$$

\bigskip

Note that $d_m \leq a^2(u) \int_M uP_{\alpha,\beta}(u)dv_g \leq a^2 (u) C_2 $.

\medskip

We have that $$\frac{1}{2}\int_M auP_{\alpha,\beta}(au)dv_g=\frac{a^{q+1}}{q+1}\int_M |u|^{q+1}dv_g.$$

Therefore,
\begin{align*}
I\left( \frac{\sqrt{d_m}}{C} u \right) =&\frac{d_m}{2C^2}\int_M uP_{\alpha,\beta}(u)dv_g-\left({\frac{\sqrt{d_m}}{C}}\right)^{q+1}\frac{a^2}{2a^{q+1}}\int_M uP_{\alpha,\beta}(u)dv_g\\
=&\left(\frac{d_m}{C^2}-\left(\frac{\sqrt{d_m}}{C}\right)^{q+1}\frac{1}{a^{q-1}}\right)  
\frac{1}{2}\int_M uP_{\alpha,\beta}(u)dv_g\\
=&\frac{d_m}{2 C^2}\left(1-\left(\frac{\sqrt{d_m}}{Ca}\right)^{q-1}\right) \int_M uP_{\alpha,\beta}(u)dv_g\\
\geq& \frac{d_m}{2 C^2}\left(1-(\sqrt{C_2}/C)^{q-1}\right) \int_M uP_{\alpha,\beta}(u)dv_g.
\end{align*}

And since we had picked $C>\sqrt{C_2 }$, it follows that 

\begin{equation}\label{B}
I\left(   \frac{\sqrt{d_m}}{C} u \right)   \geq kd_m^2,
\end{equation}

\noindent
for some constant $k>0$.

\bigskip

Now we prove:

\begin{lemma}
For each $m$ there exists $h_m\in \Gamma^*$ such that $h_m(u)=\frac{\sqrt{d_m}}{C}u$  for all $u\in E_m^{\perp}$.
\end{lemma}

\begin{proof} Fix $m$ and for any $u\in L_{2,f}^2 (M)$ write $u=u_1 +u_2$ with $u_1 \in E_m$ and 
$u_2 \in E_m^{\perp}$. Define $h_m (u) = \varepsilon u_1 + \frac{\sqrt{d_m}}{C}u_2$.

\medskip

We will prove that if $\varepsilon >0$ is small enough,  then $h_m \in \Gamma^*$. Clearly $h_m$ is an odd homeomorphism
of $L_{2,f}^{2} (M)$. Therefore we need to prove that  for $\varepsilon >0$ small enough $I(h_m (B(1)) \geq 0$.
This is equivalent to saying that  for $\varepsilon >0$ small enough

$$Z_{\varepsilon} =\varepsilon (E_m  \cap B(1) )+ \frac{\sqrt{d_m}}{C} (E_m^{\perp} \cap B(1) )  \subset
\{I>0 \} \cup \{ 0\} .$$

Assume it is not true. Then there exists a sequence $\varepsilon_i >0$, $\varepsilon_i \rightarrow 0$, and a sequence
$u_i = \varepsilon_i v_i +  \frac{\sqrt{d_m}}{C} w_i$ with $v_i \in E_m  \cap B(1)$, 
$w_i \in E_m^{\perp} \cap B(1)$ such that $I(u_i )<0$.
Both sequences $v_i$ and $w_i$ are bounded in $L_{2,f}^2$ and therefore we can assume that they are
weakly convergent in $L_{2,f}^2$ and strongly convergent in $L^{q+1} (M)$ to $v_0$, $w_0$ respectively. 
It follows that $u_i$ converges to $\frac{\sqrt{d_m}}{C} w_0$, and therefore
$I(\frac{\sqrt{d_m}}{C} w_0 )\leq 0$. But this contradicts (\ref{B}).

\end{proof}

Since $d_m \rightarrow \infty$  it follows from (\ref{B}) that $$\inf\limits_{u\in E_m^{\perp}\cap \partial B(1)}I(h_m(u))\rightarrow\infty\quad\text{as}\quad m\rightarrow\infty.$$

\bigskip
This implies that $c_m\rightarrow\infty$. Note also that for a solution of Equation (\ref{Paneitz}) we
have that 

$$ \int\limits_{M}\left((\Delta u)^2 +\alpha |\nabla u|^2+\beta u^2\right)dv_g =
\int\limits_M |u|^{q+1}dv_g. $$

Therefore 

$$I(u)= \left( \frac{1}{2}  -\frac{1}{q+1} \right) \int\limits_M |u|^{q+1}dv_g. $$

Then if $u_m$ is a solution with $I(u_m ) = c_m$ then $E(u_m ) \rightarrow \infty$, 
which finishes the proof of Theorem 1.2.


\begin{thebibliography}{90}

\bibitem{Ambrosetti} A. Ambrosetti and P. H. Rabinowitz, {\it Dual variational methods in critical point theory and applications}, J. Funct. Anal., 14, (1973) 349-381.

\bibitem{Bartsch} T. Bartsch, M. Schneider, T. Weth, {\it Multiple solutions to a critical polyharmonic equation}, 
J. Reine Angew. Math. 571 (2004), 131-143.

\bibitem{Bettiol} R. Bettiol, P. Piccione, Y. Sire, {\it Nonuniqueness of conformal metrics of constant $Q$-curvature}, 
IMRN 2021, vol 9 (2021), 6967-6992.

\bibitem{Branson} T. P. Branson,  {\it Differential operators canonically associated to a conformal structure}, 
Math. Scand., 57 (1985), 293-345.



\bibitem{Hebey} Z. Djadli, E. Hebey, M. Ledoux, {\it Paneitz-type operators and applications}, Duke Math. J. 104 (2000), 129-169.




\bibitem{Esposito} P. Esposito, F. Robert, {\it Mountain pass critical points for Paneitz-Branson operators}, Calc. Var. Partial Differential Equations 15 (2002), 493-517.


\bibitem{Hebey2} E. Hebey, F. Robert, {\it Compactness and global estimates for the geometric Paneitz equation in high dimensions}, Electron. Res. Announc. Amer. Math. Soc. 10 (2004), 135–141.




\bibitem{Ferus} D. Ferus, H. Karcher, H. F. M\"{u}nzner, {\it Cliffordalgebren und neue isoparametrische
Hyperflachen}, Math. Z. 177 (1981), 479-502.

\bibitem{Tubes} A. Gray, {\it Tubes},  Second Edition, Progress in Mathematics , Vol.221. Birkh\"auser Verlag
Basel, Boston, Berlin (2004).


\bibitem{Gursky} M. J. Gursky, F. Hang, Y.-J. Lin, {\it Riemannian manifolds with positive Yamabe invariant and Paneitz operator}, Int. Math. Res. Not. IMRN 5 (2016), 1348-1367.

\bibitem{Malchiodi} M. J. Gursky, A. Malchiodi, {\it A strong maximum principle for the Paneitz operator and a non-local flow for the $Q$-curvature}, J. Eur. Math. Soc. (JEMS) 17 (2015), no. 9, 2137-2173.

\bibitem{Hang} F. Hang,  P. C. Yang, {\it Sign of Green’s function of Paneitz operators and the Q curvature}, 
Int. Math. Res. Not. IMRN 19 (2015), 9775-9791.

\bibitem{Yang} F. Hang,  P. C. Yang, {\it $Q$-curvature on a class of manifolds with dimension at least 5}, Comm. Pure Appl. Math. 69 (2016), no. 8, 1452-1491.

\bibitem{Hebey3} E. Hebey, and M. Vaugon, {\it Sobolev spaces in the presence of symmetries},
J. Math. Pures Appl.  76 (1997),  859-881.


\bibitem{Henry} G. Henry, {\it Isoparametric functions and nodal solutions of the Yamabe equation},  Ann. of Global Anal. and Geometry  56
(2019), 203-219.

\bibitem{BQC} J. Julio-Batalla, J. Petean, {\it Global bifurcation for Paneitz type equations and  constant
Q-curvature metrics}, arXiv:2312.01226.

\bibitem{subcritical} M. Lazzo, P.G. Schmidt, {\it Oscillatory radial solutions for subcritical biharmonic equations},  J. Differential Equations 247 (2009) 1479–1504

%\bibitem{degenerate} M. Lazzo, P.G. Schmidt.{\it Periodic Solutions and Invariant Manifolds for an Even-Order Differential Equation with Power Nonlinearity}, J Dyn Diff Equat (2011) 23:141–166

\bibitem{Li} G. Li, {\it A compactness theorem on Branson’s $Q$-curvature equation}, 
Pacific J. Math. 302 (2019), no. 1, 119-179.

\bibitem{YanYanLi} Y.Y. Li,  J. Xiong, {\it Compactness of conformal metrics with constant $Q$-curvature. I}, 
Adv. Math. 345 (2019), 116-160.

\bibitem{Lin} C.-S. Lin, {\it A classification of solutions of a conformally invariant fourth order equation in ${\bf R}^n$}, 
Comment. Math. Helv. 73 (1998), 206-231.

\bibitem{Malchiodi2} A. Malchiodi, {\it Compactness of solutions to some geometric fourth-order equations}, 
J. Reine Angew. Math. 594 (2006), 137-174.

\bibitem{Miyaoka} R. Miyaoka, {\it Transnormal functions on a Riemannian manifold}, Diff. Geometry and its Appl. 31 (2013) 130-139. 

\bibitem{Ozeki} H. Ozeki, M. Takeuchi, {\it On some types of isoparametrics hypersurfeces in spheres
I}, Tohoku Math. Journ. 27, (1975), 515-559.

\bibitem{Takeuchi} H. Ozeki, M. Takeuchi, {\it On some types of isoparametrics hypersurfeces in spheres
II}, Tohoku Math. Journ. 28, (1976), 7-55.

\bibitem{Paneitz} S. Paneitz, {\it A quartic conformally covariant differential operator for arbitrary pseudo-Riemannian manifolds},
preprint (1983).

\bibitem{Qian} C. Qian, Z. Tang, {\it Isoparametric functions on exotic spheres}, Adv. Math. 272 (2015), 611-629.

\bibitem{Qing} J. Qing, D. Raske, {\it On positive solutions to semilinear conformally invariant equations on locally conformally flat manifolds}, Int. Math. Res. Not. Art. ID 94172 (2006).

\bibitem{Robert} F. Robert, {\it Positive solutions for a fourth-order equation invariant under isometries}, Proc. Amer. Math.
Soc. 131 (2002), 1423-1431.

\bibitem{Saintier} N. Saintier, {\it Changing sign solutions of a conformally invariant fourth-order equation in the Euclidean space},  Comm. in Analysis and Geometry 14 (2006) 613-624.

\bibitem{Vetois} J. Vetois, {\it Uniqueness of conformal metrics with constant $Q$-curvature on closed Einstein manifolds}, 
Potential Anal. (2023), https://doi.org/10.1007/s11118-023-10117-1.

\bibitem{Wang} Q-M. Wang, {\it Isoparametric functions on Riemannian manifolds}, Math. Ann. 277 (1987), 639-646.

\bibitem{class}  J. Wei, X. Xu, {\it Classification of solutions of higher order conformally invariant equations}, Math. Ann. 313
(1999), 207–228.

\bibitem{Wei} J. Wei, C. Zhao, {\it Non-compactness of the prescribed $Q$-curvature problem in large dimensions}, 
Calc. Var. Partial Differential Equations 46 (2013), no. 1-2, 123-164.







\end{thebibliography}
\end{document}